\documentclass[12pt,a4paper]{article}
\usepackage[T1]{fontenc}
\usepackage{amsfonts}
\usepackage{amssymb,amsmath,amsbsy,amsthm}
\usepackage{newtxtext,newtxmath}
\usepackage{blkarray}
\usepackage{tikz}
\usepackage{tkz-euclide}
\usepackage{tikz-cd}
\usepackage{enumerate}
\usetikzlibrary{patterns,arrows.meta,calc,positioning}
\usepackage{graphicx}
\usepackage{float}
\usepackage{cite}
\usepackage{mathtools}
\usepackage{indentfirst}
\usepackage{lineno}
\usepackage{titlesec}
\usepackage{enumitem}
\usepackage{xcolor}
\usepackage{aliascnt}
\usepackage{varioref}
\usepackage{microtype}
\usepackage{needspace}
\usepackage{xurl}
\usepackage{hyperref}
\hypersetup{
  colorlinks=true,
  linkcolor=cyan,
  anchorcolor=cyan,
  citecolor=red,
  urlcolor=black,
  pdfencoding=auto,
  pdftitle={Log-Concavity of Conic Intrinsic Volumes},
  pdfauthor={Suijie Wang},
  pdfsubject={Conic intrinsic volumes and Alexandrov--Fenchel inequalities},
  pdfkeywords={conic intrinsic volume, mixed volume, Alexandrov--Fenchel inequality,
    Lorentzian polynomial, relative log-concavity, ultra-log-concavity}
}
\usepackage[nameinlink,capitalise,noabbrev]{cleveref}

\numberwithin{equation}{section}
\newtheorem{theorem}{Theorem}[section]
\newaliascnt{lemma}{theorem}

\aliascntresetthe{lemma}

\newaliascnt{proposition}{theorem}
\newtheorem{proposition}[proposition]{Proposition}
\aliascntresetthe{proposition}

\newaliascnt{corollary}{theorem}

\aliascntresetthe{corollary}

\newaliascnt{claim}{theorem}

\aliascntresetthe{claim}

\newaliascnt{conjecture}{theorem}

\aliascntresetthe{conjecture}

\theoremstyle{definition}
\newaliascnt{definition}{theorem}

\aliascntresetthe{definition}

\newaliascnt{example}{theorem}

\aliascntresetthe{example}

\newaliascnt{notation}{theorem}

\aliascntresetthe{notation}

\newaliascnt{question}{theorem}

\aliascntresetthe{question}

\theoremstyle{remark}
\newaliascnt{remark}{theorem}

\aliascntresetthe{remark}

\allowdisplaybreaks

\tikzset{
  graphvertex/.style={circle,fill,inner sep=1.35pt},
  leafvertex/.style={circle,draw,fill=white,inner sep=1.5pt},
  facevertex/.style={circle,draw,fill=white,inner sep=1.5pt,font=\scriptsize},
  graphlabel/.style={font=\scriptsize,inner sep=1pt}
}

\newcommand{\R}{\mathbb{R}}
\newcommand{\E}{\mathbb{E}}
\newcommand{\vol}{\operatorname{Vol}}

\begin{document}
\begin{center}
{\Large\bf Log-Concavity of Conic Intrinsic Volumes}\\[7pt]
\end{center}
\vskip3mm

\begin{center}
Houshan Fu$^{1}$ \quad Suijie Wang$^{2,*}$\\[8pt]
  
$^{1}$School of Mathematics and Information Science, Guangzhou University\\
Guangzhou 510006, Guangdong, P. R. China\\[8pt]

$^{2}$School of Mathematics, Greater Bay Area Institute for Innovation, Hunan University\\
 Changsha 410082, Hunan, P. R. China\\[8pt]

$^*$Correspondence to be sent to: wangsuijie@hnu.edu.cn\\
E-mail: fuhoushan@gzhu.edu.cn\\[12pt]
\end{center}
\vskip 3mm

\begin{abstract}
Let $n\ge1$ and $C\subseteq\R^n$ be a closed convex cone with conic intrinsic volumes $v_0(C),\ldots,v_n(C)$.  We prove the long-standing log-concavity conjecture for this sequence, in the stronger form
\[
 v_k(C)^2\ge \rho_k\rho_{n-k}v_{k-1}(C)v_{k+1}(C),  \qquad 1\le k\le n-1,
\]
where, for $l\ge1$, $\rho_l=\frac{l+1}{l}\frac{\omega_{l-1}\omega_{l+1}}{\omega_l^2}>1$ and $\omega_j$ is the volume of the Euclidean unit ball in $\R^j$.  The proof applies
the Alexandrov--Fenchel inequality to the Takemura--Kuriki identity
\[
 V(A[k],D[n-k])=\frac{\omega_k\omega_{n-k}}{\binom nk}v_k(C),\qquad A=C\cap B^n,\quad D=C^\circ\cap B^n,
\]
where $C^\circ$ is the polar cone, $B^n$ is the Euclidean unit ball, and repeated arguments are indicated by brackets.  A Master Steiner argument gives the identity directly for arbitrary closed convex cones, including degenerate ones.  We also give a circular-cone counterexample to the standard ultra-log-concavity normalizations.
\end{abstract}



\section{Introduction}
The study of conic intrinsic volumes associated to convex cone, is a classical subject in convex geometry, going back at least to Sommerville \cite{Sommerville}. The conic intrinsic volumes
\[
 v_0(C),v_1(C),\ldots,v_n(C)
\]
of a closed convex cone $C\subseteq\R^n$ form a probability
distribution.  They arise in spherical Steiner formulas, conic
integral geometry, Gaussian projection theory, and the analysis of
phase transitions in random convex optimization; see
\cite{AmelunxenLotz,ALMT,McCoyTropp,SchneiderWeil}.

Amelunxen formulated the equivalent spherical log-concavity statement
as Conjecture~4.4.16 in \cite{AmelunxenThesis}.  The problem was later
recorded in the study of conic phase transitions
\cite[Section~6.1]{ALMT}, in the Dagstuhl report
\cite[Section~5.4]{AmelunxenDagstuhl}, and in Schneider's monograph
\cite[Section 4.5]{SchneiderCones}.  In conic notation, it asks whether
\[
 v_k(C)^2\ge v_{k-1}(C)v_{k+1}(C),
 \qquad 1\le k\le n-1.
\]
A 2025 preprint on concentration for the intrinsic-volume random
variable continued to list general log-concavity as open
\cite[Section~6]{FernandezMelbournePalafox}.

For $j\ge0$, write $\omega_j=\pi^{j/2}/\Gamma(j/2+1)$, with $\omega_0=1$, and let $V$ denote the usual mixed volume.  In expressions such as
$V(K[k],L[n-k])$, brackets indicate the number of repeated arguments.

Takemura and Kuriki established the relevant bridge to mixed-volume
theory in their study of chi-bar-squared distributions.  In their
standing full-dimensional setting, their result
\cite[Theorem~2.1]{TakemuraKuriki} reads, in the present notation, for
$A=C\cap B^n$ and $D=C^\circ\cap B^n$,
\begin{equation}\label{eq:intro-mixed}
 V(A[k],D[n-k])
 =\frac{\omega_k\omega_{n-k}}{\binom nk}v_k(C),
 \qquad 0\le k\le n.
\end{equation}
Their weights $w_k$ are the quantities now denoted by $v_k(C)$, and
their cone $K^*$ is the polar cone $C^\circ$ used here.  The
Master Steiner argument in Section~\ref{sec:mixed} gives a direct
derivation for arbitrary closed convex cones, including degenerate
ones.

Formula~\eqref{eq:intro-mixed} is therefore not new.  The contribution
of this note is to combine it with the Alexandrov--Fenchel inequality.
The mixed-volume sequence on the left of \eqref{eq:intro-mixed} is
log-concave, and substitution gives a strict improvement of the
conjectured inequality.  This mechanism differs from the earlier
relation in which Amelunxen expressed the Euclidean intrinsic volumes
of $C\cap B^n$ as triangular transforms of the conic intrinsic volumes
\cite[Proposition~4.4.18]{AmelunxenThesis}.  He observed that
log-concavity of the Euclidean sequence does not invert through this
transform \cite[Remarks~4.4.19--4.4.20]{AmelunxenThesis}.  In
contrast, \eqref{eq:intro-mixed} pairs $C\cap B^n$ with
$C^\circ\cap B^n$ and identifies each conic intrinsic volume with one
mixed volume.

For $l\ge1$, set
$\rho_l:=\frac{l+1}{l}
\frac{\omega_{l-1}\omega_{l+1}}{\omega_l^2}$.
As shown in Section~\ref{sec:refined}, one has $\rho_l>1$.

\begin{theorem}[Refined log-concavity of conic intrinsic volumes]
\label{thm:main-lc}
Let $n\ge1$.  Every closed convex cone $C\subseteq\R^n$ satisfies, for
$1\le k\le n-1$,
\begin{equation}\label{eq:intro-strong}
 v_k(C)^2\ge
 \rho_k\rho_{n-k}v_{k-1}(C)v_{k+1}(C).
\end{equation}
\end{theorem}

\section{Preliminaries}
Throughout, $n\ge1$, and $B^n$ denotes the closed Euclidean unit ball in $\R^n$. Thus $\omega_n$ equals the usual volume of $B^n$, i.e., $\omega_n=\vol_n(B^n)$. A set $C\subseteq \R^n$ is a \emph{convex cone} if it is closed under non-negative linear combinations. For a closed convex cone $C\subseteq\R^n$, its \emph{polar cone} is given by
\[
C^\circ:=\{y\in\R^n:\langle x,y\rangle\le0\text{ for every }x\in C\}.
\]
Let $\Pi_C$ denote Euclidean projection onto $C$.  Moreau's decomposition theorem \cite{Moreau} states that every $x\in\R^n$ has the orthogonal decomposition
\begin{equation}\label{eq:moreau}
 x=\Pi_C(x)+\Pi_{C^\circ}(x),\qquad\langle\Pi_C(x),\Pi_{C^\circ}(x)\rangle=0.
\end{equation}

The conic intrinsic volumes $v_0(C),\ldots,v_n(C)$ may be defined
through the conic Steiner formula, or equivalently by continuous
extension from polyhedral cones; see
\cite[Definition~2.2 and Section~8.2]{McCoyTropp}.  They satisfy
\[
 v_k(C)\ge0,\qquad
 \sum_{k=0}^n v_k(C)=1.
\]

We use the Master Steiner formula of McCoy and Tropp \cite[Theorem~3.1 and Corollary~3.2]{McCoyTropp}.  Let $X_0=0$, and for $j\ge1$ let
$X_j$ have the chi-squared distribution with $j$ degrees of freedom.
For every bounded Borel function $f:[0,\infty)^2\to\R$ ,
\begin{equation}\label{eq:master-Steiner}
 \E\left[f\bigl(\|\Pi_C(g)\|^2,\|\Pi_{C^\circ}(g)\|^2\bigr)\right]
 =\sum_{k=0}^n v_k(C)\E\left [f(X_k,X'_{n-k})\right],
\end{equation}
where $g\sim\mathcal{N}(0,I_n)$ is a standard Gaussian vector in $\R^n$ and the variables in each pair
$(X_k,X'_{n-k})$ are independent, with $X'_j$ an independent copy
of $X_j$.  Following \cite[(5.28)]{SchneiderConvexBodies}, we
normalize mixed volumes by
\begin{equation}\label{eq:mixed-expansion}
 \vol_n(tK+sL)
 =\sum_{k=0}^n\binom nk
 V(K[k],L[n-k])t^ks^{n-k}
\end{equation}
for compact convex sets $K,L\subseteq\R^n$ and $s,t\ge0$.
The Alexandrov--Fenchel inequality
\cite[Theorem 7.3.1]{SchneiderConvexBodies} is
\begin{equation}\label{eq:AF-prelim}
 V(K[k],L[n-k])^2\ge V(K[k-1],L[n-k+1])V(K[k+1],L[n-k-1]),\quad 1\le k\le n-1.
\end{equation}
\section{The Takemura--Kuriki representation}\label{sec:mixed}
We recall the set identity and mixed-volume formula of Takemura and
Kuriki \cite[Lemma~2.1 and Theorem~2.1]{TakemuraKuriki}.  Their formula was stated under the standing nonempty-interior assumption of their
statistical setting.  The Master Steiner argument below supplies a short derivation in closed cone generality, including all lower-dimensional and non-pointed cases; no novelty is claimed for the representation itself.

Fix a closed convex cone $C\subseteq\R^n$ and set $A=C\cap B^n$ and $D:=C^\circ\cap B^n$. For every $s,t\ge0$, Takemura and Kuriki's identity
\cite[Lemma~2.1]{TakemuraKuriki} gives
\begin{equation}\label{eq:set-identity}
 tA+sD
 =(C+sB^n)\cap(C^\circ+tB^n)
 =\{x\in\R^n:\|\Pi_C(x)\|\le t,
                  \ \|\Pi_{C^\circ}(x)\|\le s\}.
\end{equation}
For completeness, the identity holds without a dimensional
assumption.  Indeed, if $x=a+b$ with $a\in tA$ and $b\in sD$, then
$x\in C+sB^n$ and $x\in C^\circ+tB^n$.  Conversely, Moreau's
decomposition and the distance identities
\[
 \operatorname{dist}(x,C)=\|\Pi_{C^\circ}(x)\|,
 \qquad
 \operatorname{dist}(x,C^\circ)=\|\Pi_C(x)\|
\]
show that membership in the intersection on the right of
\eqref{eq:set-identity} is equivalent to
$\|\Pi_Cx\|\le t$ and $\|\Pi_{C^\circ}x\|\le s$.  In that case,
\eqref{eq:moreau} expresses $x$ as an element of $tA+sD$.

\begin{proposition}[Takemura--Kuriki formula in closed cone form]
For all $s,t\ge0$,
\begin{equation}\label{eq:intro-volume}
 \vol_n(tA+sD)=\sum_{k=0}^n
 \omega_k\omega_{n-k}v_k(C)t^ks^{n-k}.
\end{equation}
Consequently, \eqref{eq:intro-mixed} holds for $0\le k\le n$.
\end{proposition}

\begin{proof}
Both sides of \eqref{eq:intro-volume} are continuous on $[0,\infty)^2$, the left-hand side by Hausdorff continuity of volume. It therefore suffices to take $s,t>0$. Consider the bounded Borel function
\[
 f_{s,t}(a,b)=e^{(a+b)/2}\mathbf 1_{\{a\le t^2,\ b\le s^2\}}, \quad (a,b)\in[0,\infty)^2.
\]
By the orthogonality in \eqref{eq:moreau},
\[
 \|\Pi_C(x)\|^2+\|\Pi_{C^\circ}(x)\|^2=\|x\|^2.
\]
Together with \eqref{eq:set-identity},  this yields
\begin{equation}\label{eq:master-left}
 \E\!\left[f_{s,t}\bigl(\|\Pi_C(g)\|^2, \|\Pi_{C^\circ}(g)\|^2\bigr)\right]=(2\pi)^{-n/2}\vol_n(tA+sD).
\end{equation}
For every integer $d\ge0$ and $r>0$, direct integration of the chi-squared density gives
\begin{equation}\label{eq:tilted-chi}
 \E\left[e^{X_d/2}\mathbf 1_{\{X_d\le r^2\}}\right]
 =\frac{r^d}{2^{d/2}\Gamma(d/2+1)},
\end{equation}
with value $1$ when $d=0$.  Apply the Master Steiner formula \eqref{eq:master-Steiner} to $f_{s,t}$ and use the independence of $X_k$ and $X'_{n-k}$.  Equations \eqref{eq:master-left} and
\eqref{eq:tilted-chi} yield
\[
 (2\pi)^{-n/2}\vol_n(tA+sD)=\sum_{k=0}^n v_k(C)\frac{t^ks^{n-k}}{2^{n/2}\Gamma(k/2+1)\Gamma((n-k)/2+1)}.
\]
Multiplication by $(2\pi)^{n/2}$ proves
\eqref{eq:intro-volume}.  Comparing coefficients with
\eqref{eq:mixed-expansion} proves \eqref{eq:intro-mixed}.
\end{proof}
\section{Strengthened log-concavity and its limits}\label{sec:refined}

We first verify the strict factor in Theorem~\ref{thm:main-lc}.  Set
$b_j=j!\omega_j$.  The duplication formula gives
\[
 b_j=2^j\pi^{(j-1)/2}
 \Gamma\left(\frac{j+1}{2}\right).
\]
Strict log-convexity of the Gamma function yields
$b_j^2<b_{j-1}b_{j+1}$.  Since
\[
 \frac{b_{j-1}b_{j+1}}{b_j^2}
 =\frac{j+1}{j}
  \frac{\omega_{j-1}\omega_{j+1}}{\omega_j^2}
 =\rho_j,
\]
one has $\rho_j>1$ for every $j\ge1$.

Applying the Alexandrov--Fenchel inequality to the coefficients in \eqref{eq:intro-volume} yields the conjectured inequalities in Theorem \ref{thm:main-lc}.
\begin{proof}[Proof of Theorem \ref{thm:main-lc}]
Fix a closed convex cone $C\subseteq\R^n$, set
$A=C\cap B^n$ and $D=C^\circ\cap B^n$, and write
$W_k=V(A[k],D[n-k])$.  To include the possibility that $A$ or $D$ is
lower-dimensional, for $\varepsilon>0$ set
\[
 A_\varepsilon=A+\varepsilon B^n,
 \qquad D_\varepsilon=D+\varepsilon B^n.
\]
These are convex bodies, so the Alexandrov--Fenchel inequality applies
to them.  Letting $\varepsilon\downarrow0$ and using continuity of
mixed volumes in the Hausdorff metric
\cite[Section~5.1]{SchneiderConvexBodies} gives
\begin{equation}\label{eq:AF}
 W_k^2\ge W_{k-1}W_{k+1}.
\end{equation}
Substitution of \eqref{eq:intro-mixed} into \eqref{eq:AF} gives
\begin{align*}
 v_k(C)^2
 &\ge
 \frac{\binom nk^2}
 {\binom n{k-1}\binom n{k+1}}
 \frac{\omega_{k-1}\omega_{k+1}}{\omega_k^2}
 \frac{\omega_{n-k-1}\omega_{n-k+1}}
      {\omega_{n-k}^2}
 v_{k-1}(C)v_{k+1}(C)\\
 &=\rho_k\rho_{n-k}
 v_{k-1}(C)v_{k+1}(C).
\end{align*}
Here we used
$\binom nk^2/[\binom n{k-1}\binom n{k+1}]
=(k+1)(n-k+1)/[k(n-k)]$.
This proves \eqref{eq:intro-strong}.
\end{proof}

\medskip
\noindent\emph{Limits of ultra-log-concavity.}
Both standard ultra-log-concavity normalizations can fail.  It is
enough to test the weaker Poisson condition, which for a sequence
$(a_k)_{k=0}^n$ with interval support requires
\[
 a_k^2\ge
 \left(1+\frac1k\right)a_{k-1}a_{k+1}.
\]
The binomial ULC$(n)$ condition contains the additional factor
$1+1/(n-k)$ and is therefore stronger; see
\cite{MarsigliettiMelbourne}.  Let
\[
 \operatorname{Circ}_d(\alpha)
 =\{x\in\R^d:x_1\ge\|x\|\cos\alpha\},
 \qquad 0<\alpha<\frac\pi2.
\]
The conic intrinsic volumes of circular cones are given explicitly in
\cite[Appendix~D.1]{ALMT}.

Take $d=4$ and $\alpha=\pi/3$.  The entries needed at $k=1$ are
\[
 v_0=\frac16-\frac{\sqrt3}{4\pi},
 \qquad
 v_1=\frac18,
 \qquad
 v_2=\frac{\sqrt3}{2\pi}.
\]
The Poisson condition would require $v_1^2\ge2v_0v_2$, but
\[
 2v_0v_2-v_1^2
 =\frac{32\sqrt3\,\pi-3\pi^2-144}{192\pi^2}>0.
\]
Indeed, using $333/106<\pi<355/113$ and $265/153<\sqrt3$ gives
\[
 \pi(32\sqrt3-3\pi)-144
 >\frac{110453}{217073}>0.
\]
Hence both conditions fail.  Since conic intrinsic volumes are
continuous under conic Hausdorff convergence and circular cones admit
polyhedral approximations
\cite[Fact~7.4 and Proposition~8.2]{McCoyTropp}, the same strict violations occur
for all sufficiently close polyhedral approximants.

\section*{Declaration on the use of generative AI}
During the preparation of this manuscript, the authors used ChatGPT  primarily to assist with literature searches and identify relevant prior work. The authors carefully examined all cited sources, verified their relevance and mathematical content, and take full responsibility for all arguments, citations, and conclusions presented in this paper.

\section*{Acknowledgements}
This work is supported by the National Natural Science Foundation of China (Grant No. 12571350) and the Guangdong Basic and Applied Basic Research Foundation (Grant No. 2026A1515012237, Grant No. 2025A1515010457). 
\begingroup
\hbadness=10000

\endgroup


\begin{thebibliography}{99}
\bibitem{AmelunxenThesis}
D.~Amelunxen,
\emph{Geometric Analysis of the Condition of the Convex Feasibility
Problem}, Ph.D. thesis, Universität Paderborn, 2011.

\bibitem{AmelunxenDagstuhl}
D.~Amelunxen,
Conjecture: log-concavity of conic intrinsic volumes,
in P.~B\"urgisser, F.~Cucker, M.~Karpinski, and N.~Vorobjov (eds.),
\emph{Complexity of Symbolic and Numerical Problems
(Dagstuhl Seminar 15242)},
\emph{Dagstuhl Rep.} \textbf{5} (2016), no.~6, pp. 44.


\bibitem{AmelunxenLotz}
D.~Amelunxen and M.~Lotz,
Intrinsic volumes of polyhedral cones: a combinatorial perspective,
\emph{Discrete Comput. Geom.} \textbf{58} (2017), 371--409.

\bibitem{ALMT}
D.~Amelunxen, M.~Lotz, M.~B. McCoy, and J.~A. Tropp,
Living on the edge: phase transitions in convex programs with random
data,
\emph{Inf. Inference} \textbf{3} (2014), 224--294.

\bibitem{FernandezMelbournePalafox}
M.~Fern\'andez-Unzueta, J.~Melbourne, and G.~Palafox-Castillo,
On convex functions of Gaussian variables,
arXiv:2510.06676, 2025.

\bibitem{LinLindsay}
Y.~Lin and B.~G. Lindsay,
Projections on cones, chi-bar squared distributions, and Weyl's formula,
\emph{Statist. Probab. Lett.} \textbf{32} (1997), no.~4, 367--376.

\bibitem{MarsigliettiMelbourne}
A.~Marsiglietti and J.~Melbourne,
Concentration inequalities for log-concave sequences,
\emph{Int. Math. Res. Not.} (2026), no.~4, rnag023.

\bibitem{McCoyTropp}
M.~B. McCoy and J.~A. Tropp,
From Steiner formulas for cones to concentration of intrinsic volumes,
\emph{Discrete Comput. Geom.} \textbf{51} (2014), 926--963.

\bibitem{Moreau}
J. J.~Moreau,
D\'ecomposition orthogonale d'un espace hilbertien selon deux c\^ones
mutuellement polaires,
\emph{C. R. Acad. Sci. Paris} \textbf{255} (1962), 238--240.

\bibitem{SchneiderConvexBodies}
R.~Schneider,
Convex Bodies: The Brunn--Minkowski Theory,
2nd expanded ed., Encyclopedia of Mathematics and its Applications,
vol.~151, Cambridge University Press, Cambridge, 2014.

\bibitem{SchneiderCones}
R.~Schneider,
Convex Cones: Geometry and Probability,
Lecture Notes in Mathematics, vol.~2319,
Springer, Cham, 2022.

\bibitem{SchneiderWeil}
R.~Schneider and W.~Weil,
Stochastic and Integral Geometry,
Probability and its Applications, Springer, Berlin, 2008.

\bibitem{Sommerville}
D. M. Y. Sommerville, The relations connecting the angle-sums and volume of a polytope in
space of n dimensions, \emph{Proc. R. Soc. Lond. A} \textbf{115} (1927), no 770, 103--119.

\bibitem{TakemuraKuriki}
A.~Takemura and S.~Kuriki,
Weights of $\bar\chi^2$ distribution for smooth or piecewise smooth
cone alternatives,
\emph{Ann. Statist.} \textbf{25} (1997), no.~6, 2368--2387.

\end{thebibliography}
\end{document}